\documentclass{amsart}
\usepackage{amsmath,amsfonts,amssymb,enumerate}
\usepackage{color}

\newcommand{\N}{\mathbb{N}}                   
\newcommand{\R}{\mathbb{R}}                   
\newcommand{\C}{\mathbb{C}}                   
\newcommand{\vp}{\varphi}                     
\newcommand{\OO}{\mathcal O}                  
\newcommand{\ccon}{\rightarrowtail}           
\newcommand{\Reg}{\mathrm{Reg}}               
\newcommand{\Sing}{\mathrm{Sing}}             
\newcommand{\tX}{\tilde{X}}
\newcommand{\hX}{\hat{X}}
\newcommand{\tY}{\tilde{Y}}

\newcommand{\hZ}{\hat{Z}}
\newcommand{\NA}{\mathcal{N}\!\mathcal{A}}

\newtheorem{theorem}{Theorem}[section]
\newtheorem{proposition}[theorem]{Proposition}
\newtheorem{lemma}[theorem]{Lemma}
\newtheorem{corollary}[theorem]{Corollary}

\theoremstyle{definition}

\newtheorem{example}[theorem]{Example}
\newtheorem{remark}[theorem]{Remark}

\numberwithin{equation}{section}

\begin{document}
\title[Nash approximation of complex analytic sets in Runge domains]
 {On Nash approximation of complex analytic sets in Runge domains}

\author{Janusz Adamus}
\address{J. Adamus, Department of Mathematics, The University of Western Ontario, London, Ontario N6A 5B7 Canada
         -- and -- Institute of Mathematics, Faculty of Mathematics and Computer Science, Jagiellonian University,
         ul. {\L}ojasiewicza 6, 30-348 Krak{\'o}w, Poland}
\email{jadamus@uwo.ca}
\author{Marcin Bilski}
\address{M. Bilski, Institute of Mathematics, Faculty of Mathematics and Computer Science, Jagiellonian University,
         ul. {\L}ojasiewicza 6, 30-348 Krak{\'o}w, Poland}
\email{Marcin.Bilski@im.uj.edu.pl}
\thanks{The research was partially supported by Natural Sciences and Engineering Research Council of Canada (J.\,Adamus)
 and NCN (M.\,Bilski)}
\subjclass[2010]{}

\begin{abstract}
We prove that every complex analytic set $X$ in a Runge domain $\Omega$ can be approximated
by Nash sets on any relatively compact subdomain $\Omega_0$ of $\Omega$. Moreover, for every
Nash subset $Y$ of $\Omega$ with $Y\subset X$, the approximating sets can be chosen so that
they contain $Y\cap\Omega_0$. As a consequence, we derive a necessary and sufficient
condition for a complex analytic set $X$ to admit a Nash approximation which coincides with
$X$ along its arbitrary given subset.
\end{abstract}

\maketitle

\section{Introduction and main results}
\label{sec:intro}

A basic problem in complex analysis is to approximate holomorphic maps by algebraic ones.
Classical results here are the Runge approximation theorem and the Oka-Weil approximation
theorem. These theorems have been generalized in many directions (see \cite{DLS}, \cite{DSK},
\cite{FoLa}, \cite{GA}, \cite{Ku}, \cite{Lem}, \cite{Lev}, \cite{TA}, \cite{TT-RMC} and
references therein).

The problem of algebraic approximation of holomorphic maps has a natural generalization in
complex analytic geometry. Namely, one can ask whether a complex analytic set can be
approximated by algebraic sets or by branches of algebraic sets (called Nash sets). In the
case when the analytic set under consideration has only isolated singularities (or, even, is
non-singular) the latter question is closely related to the classical problem of transforming
an analytic set onto a Nash one by a biholomorphic map (see \cite{A}, \cite{DLS}, \cite{Lem},
\cite{St}, \cite{TT1}).

The situation is quite different when the singular locus is of higher dimension. Then there
exist germs of complex analytic sets which are not biholomorphically equivalent to any germ
of a Nash set (see \cite{Wh2}). In general, only topological equivalence of analytic and Nash
set germs holds true (cf. \cite{Mo}). Nevertheless, by \cite{B-CR} and \cite{B-CA}, every
analytic subset $X$ of $D\times\mathbf{C}^p$ of pure dimension $n$ with proper projection
onto $D$, where $D\subset\C^n$ is a Runge domain, can be approximated by (branches of)
algebraic sets of pure dimension $n$. (This immediately implies that every analytic set
admits a local algebraic approximation.)
\medskip

The first goal of the present paper is to prove the following fundamental theorem, which
strengthens the approximation results of \cite{B-CR} by showing that it is not necessary to
require that $X$ have a proper projection onto some Runge domain in $\C^n$. Moreover, for
every Nash subset $Y$ contained in $X$, the approximating sequence $(X_\nu)_{\nu\in\N}$ can
be chosen so that $Y\subset X_\nu$ for all $\nu$. (See Section~\ref{sec:prelim} for the
definitions and basic properties of the notions used below.)

\begin{theorem}
\label{thm:main} Let $\Omega$ be a Runge domain in $\C^q$, and let $X$ be a complex analytic
subset of $\Omega$ of pure dimension $n$. Let $Y$ be a Nash subset of $\Omega$ such that
$Y\subset X$. Then, for every open $\Omega_0$ relatively compact in $\Omega$, there exists a
sequence $(X_\nu)_{\nu\in\N}$ of Nash subsets of $\Omega_0$ of pure dimension $n$ converging
to $X\cap\Omega_0$ in the sense of holomorphic chains and such that $X_\nu\supset
Y\cap\Omega_0$ for all $\nu$.
\end{theorem}

Given a complex analytic set $X$ in a Runge domain, denote by $\NA(X)$ the 
class of those sets $R\subset X$ for which $X$ admits a Nash approximation along $R$.
Theorem~\ref{thm:main} immediately implies that $R\in\NA(X)$ provided $R$ is
contained in a Nash set $Y$ with $Y\subset X$. A natural question is whether the latter condition
characterizes the class $\NA(X)$. In Section~\ref{sec:obstruction} we show that this
is indeed the case. Moreover, using the results of \cite{AR}, we show that the class
$\NA(X)$ contains all the semialgebraic sets contained in $X$ (a family strictly
larger than that of Nash sets contained in $X$).

\begin{theorem}
\label{thm:global-rel} Let $\Omega$ be a Runge domain in $\C^q$. Let $X$ be a complex
analytic subset of $\Omega$ of pure dimension $n$, and let $R$ be an arbitrary subset of $X$.
The following conditions are equivalent:
\begin{itemize}
\item[(i)] For every open $\Omega_0$ relatively compact in $\Omega$, there exists a sequence
$(X_\nu)_{\nu\in\N}$ of Nash subsets of $\Omega_0$ of pure dimension $n$ converging to
$X\cap\Omega_0$ in the sense of holomorphic chains and such that $X_\nu\supset R\cap\Omega_0$ for all $\nu$.
\item[(ii)] For every open $\Omega_0$ relatively compact in $\Omega$, there exists a sequence
$(X_\nu)_{\nu\in\N}$ of Nash subsets of $\Omega_0$ of pure dimension $n$ converging to
$X\cap\Omega_0$ locally uniformly and such that $X_\nu\supset R\cap\Omega_0$ for all $\nu$.
\item [(iii)] For every open $\Omega_0$ relatively compact in $\Omega$,
 there exists a semialgebraic set $S$ such that $R\cap\Omega_0\subset S\subset X$.
\item[(iv)] For every point $a$ in the Euclidean closure of $R$ in $\Omega$,
 there is a semialgebraic germ $S_a$ such that $R_a\subset S_a\subset X_a$.
\item[(v)] For every open $\Omega_0$ relatively compact in $\Omega$, there exists a Nash set $Y$ in $\Omega_0$ such that $R\cap\Omega_0\subset Y\subset X$.
\item[(vi)] For every point $a$ in the Euclidean closure of $R$ in $\Omega$,
 there is a Nash germ $Y_a$ such that $R_a\subset Y_a\subset X_a$.
\end{itemize}
\end{theorem}

The equivalence of conditions (ii) and (vi) above is related to the so-called holomorphic
closure of $R$ at a point. Holomorphic closure of real analytic and semialgebraic sets has
been studied in \cite{AR} and \cite{AS} and has been used to develop CR geometry in the
singular setting. The above theorem can be regarded as a characterization of germs $R_a$
whose holomorphic closure is contained in a Nash subgerm of $X_a$. Namely, the latter holds
precisely when $X$ can be approximated locally uniformly along $R$ by Nash sets in a
neighbourhood of $a$.
\medskip

In the following section, we recall basic definitions and facts used throughout the paper.
Theorems~\ref{thm:main} and~\ref{thm:global-rel} will be proved in Sections~\ref{sec:main}
and~\ref{sec:obstruction}, respectively.
Results on Nash approximation of analytic sets can be used, in particular,
for approximation of holomorphic functions. We give an application of this kind in the last section.

\section{Preliminaries}
\label{sec:prelim}

\subsection{Analytic subsets of domains of holomorphy}
\label{sec:ansubdh}

The following proposition is well known (see, e.g., \cite{Ho}, p. 192).

\begin{proposition}
\label{zhorm} Let $\Omega\subset\C^q$ be a domain of holomorphy, let $A$ be an analytic
subset of $\Omega$, and let $a\in\Omega\setminus A$. Then, there exists $f\in\OO(\Omega)$
such that $A\subset f^{-1}(0)$ and $f(a)\neq 0$.
\end{proposition}

As an immediate consequence of Proposition~\ref{zhorm}, one obtains that analytic subsets of
a domain of holomorphy can be described by global analytic functions.

\begin{theorem}
\label{zho22}
Let $\Omega\subset\C^q$ be a domain of holomorphy and let $A$ be an analytic subset of $\Omega$. Then, there are
$f_1,\dots,f_m\in\OO(\Omega)$ such that $A= \{z\in\Omega:f_1(z)=\ldots=f_m(z)=0\}$.
\end{theorem}

\subsection{Runge domains and polynomial polyhedra}
\label{sec:polyhed}

A domain of holomorphy $\Omega\subset\C^q$ is called a \emph{Runge domain} if every function
$f\in\OO(\Omega)$ can be uniformly approximated on every compact subset of $\Omega$
by polynomials in $q$ complex variables.

We say that a set $P$ is a \emph{polynomial polyhedron} in $\C^q$ if there exist polynomials
$f_1,\ldots,f_s\in\C[Z_1,\dots,Z_q]$ and real constants $c_1,\ldots,c_s$ such that
\[
P=\{z\in\C^q:|f_1(z)|\leq c_1,\ldots,|f_s(z)|\leq c_s\}\,.
\]

Let us recall a straightforward consequence of Theorem 2.7.3 and Lemma 2.7.4 from \cite{Ho}.

\begin{theorem}
\label{hor1} Let $\Omega\subset\C^q$ be a Runge domain. Then, for every open
$\Omega_0\Subset\Omega$ there exists a compact polynomial polyhedron $P\subset\Omega$ such
that $\Omega_0\subset P$.
\end{theorem}

(Here and throughout, we use the shorthand notation $A\Subset B$ to indicate that $A$ is a
relatively compact subset of $B$.)

\subsection{A division theorem}
\label{sec:divthjp}

The following result is a simplified version of a division theorem of \cite{JP}.

\begin{theorem}[{cf. \cite[Cor.\,2.10.3]{JP}}]
\label{divisio}
Let $\Omega\subset\C^q$ be a domain of holomorphy.
Let $F=(F_1,\ldots,F_N):\Omega\to\C^N$ be a holomorphic mapping with $F\neq0$ and let $\mu=\min\{q, N-1\}$.
Then, for any $G\in\OO(\Omega)$ with
$\| |G|\cdot \|F\|^{-(2\mu+1)} \|_{L^2(\Omega)}<+\infty$ there exist
$u_1,\ldots,u_N\in\OO(\Omega)$ such that $G=u_1F_1+\ldots+u_NF_N$.
\end{theorem}

\begin{corollary}
\label{nstzatz}
Let $D$ be any domain in $C^q$ and let $\Omega\Subset D$ be a domain of holomorphy.
Let $F=(F_1,\ldots,F_N):D\to\C^N$ be a holomorphic mapping with $F\neq 0$.
Then, for any $G\in\OO(D)$ with $\{F_1=\ldots=F_N=0\}\subset\{G=0\}$, there are
$m\in\N$ and $u_1,\ldots,u_N\in\OO(\Omega)$ such that $G^m=u_1F_1+\ldots+u_NF_N$ on $\Omega$.
\end{corollary}

\begin{proof}
The facts that $\{F_1=\ldots=F_N=0\}\subset\{G=0\}$ and $\Omega$ is relatively
compact in $D$ immediately imply that, for $m\in\N$ large enough,
$\| |G^m|\cdot \|F\|^{-(2\mu+1)} \|_{L^2(\Omega)}<+\infty,$ where $\mu=\min\{q, N-1\}$.
Hence the assertion follows from Theorem~\ref{divisio}.
\end{proof}

\subsection{Nash sets}
\label{sec:prelnash}

By a Nash set (resp. germ, function, etc.) we shall always mean a \emph{complex} Nash set
(resp. germ, function, etc.), in the following sense. Let $\Omega$ be an open subset of
$\C^q$ and let $f$ be a holomorphic function on $\Omega$. We say that $f$ is a \emph{Nash
function} at $\zeta\in\Omega$ if there exist an open neighbourhood $U$ of $\zeta$ in $\Omega$
and a polynomial $P\in\C[Z_1,\dots,Z_q,W]$, $P\neq0$, such that $P(z,f(z))=0$ for $z\in U$. A
holomorphic function on $\Omega$ is a Nash function if it is a Nash function at every point
of $\Omega$. A holomorphic mapping $\vp:\Omega\to\C^N$ is a \emph{Nash mapping} if each of
its components is a Nash function on $\Omega$.

A subset $X$ of $\Omega$ is called a \emph{Nash subset} of $\Omega$ if for every
$\zeta\in\Omega$ there exist an open neighbourhood $U$ of $\zeta$ in $\Omega$ and Nash
functions $f_1,\dots,f_s$ on $U$, such that $X\cap U=\{z\in U: f_1(z)=\dots=f_s(z)=0\}$. A
germ $X_\zeta$ of a set $X$ at $\zeta\in\Omega$ is a \emph{Nash germ} if there exists an open
neighbourhood $U$ of $\zeta$ in $\Omega$ such that $X\cap U$ is a Nash subset of $U$.
Equivalently, $X_\zeta$ is a Nash germ if its defining ideal can be generated by convergent
power series algebraic over the polynomial ring $\C[Z_1,\dots,Z_q]$. A detailed exposition of
complex Nash sets and mappings can be found in \cite{Tw}. Let us only recall here a useful
characterisation of irreducible Nash sets:

\begin{theorem}[{cf. \cite[Thm.\,2.10, Thm.\,2.11]{Tw}}]
\label{chana}
Let $X$ be an irreducible Nash subset of an open set $\Omega\subset\C^q$. Then, there
exists an algebraic subset $Z$ of $\C^q$ such that $X$ is an analytic irreducible
component of $Z\cap\Omega$. Conversely, every analytic irreducible component of $Z\cap\Omega$
is an irreducible Nash subset of $\Omega$.
\end{theorem}

(Since irreducible components of a Nash set are Nash, an irreducible Nash set is simply a
Nash set which is irreducible as a complex analytic set.)

\subsection{Semialgebraic sets}
\label{sec:semialg}

Identifying $\C^q$ with $\R^{2q}$, one can speak of semialgebraic sets and functions in
$\C^q$. Quite generally, let $M$ be a finite-dimensional $\R$-vector space. A choice of base
for $M$ gives a linear isomorphism $\psi:\R^m\to M$, where $m=\dim{M}$. We say that a
function $f:M\to\R$ is a \emph{polynomial function} on $M$ if there exists
$P\in\R[X_1,\dots,X_m]$ such that $(f\circ\psi)(x)=P(x)$ for all $x=(x_1,\dots,x_m)\in\R^m$.
Since linear base change is a polynomial mapping (with polynomial inverse), it follows that
the above definition is independent of the choice of base for $M$. We say that a subset $S$
of $M$ is \emph{semialgebraic} if $S$ is a finite union of sets of the form
\[
\{x\in M: f_1(x)=\dots=f_r(x)=0,g_1(x)>0,\dots,g_s(x)>0 \}\,,
\]
where $r,s\in\N$ and $f_1,\dots,f_r,g_1,\dots,g_s$ are polynomial functions on $M$. One
easily checks that the union and intersection of two semialgebraic sets are semialgebraic, as
is the complement of a semialgebraic set.

Let $\Omega$ and $\Delta$ be open subsets of finite-dimensional $\R$-vector spaces $M$ and
$N$ respectively. A mapping $\vp:\Omega\to\Delta$ is called a \emph{semialgebraic mapping} if
its graph is a semialgebraic subset of $M\times N$. The Tarski-Seidenberg Theorem (see, e.g.,
\cite[Prop.\,2.2.7]{BCR}) ensures that the image (resp. the inverse image) by $\vp$ of a
semialgebraic subset of $M$ (resp. $N$) is semialgebraic in $N$ (resp. $M$). Apart from the
above facts, we will use the following properties of semialgebraic sets:

\begin{remark}
\label{rem:SAG-facts}
{~}
\begin{enumerate}
\label{rem:SA-facts}

\item \cite[Prop.\,2.2.2]{BCR}. If $S$ is semialgebraic in $M$, then the topological closure and
interior of $S$ in $M$ are semialgebraic sets.

\item \cite[Thm.\,2.9.10]{BCR}. Every semialgebraic subset of $M$ is a disjoint union of a finite family of sets,
each of which is a connected real analytic manifold and a semialgebraic subset of $M$.
\end{enumerate}
\end{remark}

For a concise introduction to semialgebraic geometry, we refer the reader to \cite[Ch.\,2]{BCR} and \cite{Cos}.

\subsection{Approximation of holomorphic maps into algebraic varieties}
\label{sec:ahmiav}

The following approximation theorem is due to Lempert \cite{Lem}.

\begin{theorem}[{cf. \cite[Thm.\,3.2]{Lem}}]
\label{lat}
Let $\Omega\subset\C^q$ be a Runge domain and let $\Omega_0\Subset\Omega$ be an open set.
Let $f:\Omega\to\C^p$ be a holomorphic map that satisfies a system of equations $P(z,f(z))=0$, for every $z\in\Omega$,
where $P:\C^q\times\C^p\to\C^s$ is a polynomial map.
Then $f|_{\Omega_0}$ can be uniformly approximated by a Nash map
$F:\Omega_0\to\C^p$ satisfying $P(z,F(z))=0$ for every $z\in\Omega_0$.
\end{theorem}

\subsection{Convergence of closed sets and holomorphic chains}
\label{sec:holchai}

Let $U$ be an open set in $\C^q$. By a \emph{holomorphic chain} in $U$ we mean a formal
sum $A=\sum_{j\in J}\alpha_jC_j,$ where $\alpha_j$ are nonzero integers and
$\{C_j\}_{j\in J}$ is a locally finite family of pairwise distinct irreducible analytic
subsets of $U$ (see \cite{Ch}, \cite{Tw2}; cf. \cite{Ba}). The set $\bigcup_{j\in J}C_j$
is called the \emph{support} of $A$ and is denoted by $|A|$, whereas the $C_j$ are called the
\emph{components} of $A$ with \emph{multiplicities} $\alpha_j$. The chain $A$ is called \emph{positive} if
$\alpha_j>0$ for all $j\in J.$ If all the components of $A$ have the same dimension $n$ then
$A$ is called an \emph{$n$-chain}.

Below we introduce convergence in the sense of holomorphic chains in $U$. To do this, we will need first the
notion of the local uniform convergence of closed sets:
Let $X$ and $\{X_\nu\}_{\nu\in\N}$ be closed subsets of $U$. We say that the sequence $(X_\nu)$ \emph{converges to
$X$ locally uniformly} when the following two conditions hold:
\begin{itemize}
\item[(l1)] For every $a\in X$ there exists a sequence $(a_\nu)$ such that $a_\nu\in X_\nu$ and
$a_\nu$ converges to $a$ in the Euclidean topology on $\C^q$.
\item[(l2)] For every compact subset $K$ of $U$ such that $K\cap X=\varnothing$,
one has $K\cap X_\nu=\varnothing$ for almost all $\nu$.
\end{itemize}
Then we write $X_\nu\to X$. (For the topology of local uniform convergence,
see \cite{TPW}.)

Let now $Z$ and $\{Z_\nu\}_{\nu\in\N}$ be positive $n$-chains in $U$.
We say that the sequence $(Z_\nu)$ \emph{converges to $Z$}, when:
\begin{itemize}
\item[(c1)] $|Z_\nu|\to |Z|$, and
\item[(c2)] For every regular point $a$ of $|Z|$ and every submanifold
$T$ of $U$ of dimension $q-n$ transversal to $|Z|$ at $a$ and such that $\overline{T}$ is
compact and $|Z|\cap\overline{T}=\{a\}$, one has $deg(Z_\nu\cdot T)=deg(Z\cdot T)$ for almost
all $\nu$.
\end{itemize}
Then we write $Z_\nu\ccon Z$. (By $Z\cdot T$ we denote the intersection product of $Z$ and
$T$ (see, e.g., \cite{Tw2}). Observe that the chains $Z_\nu\cdot T$ and $Z\cdot T$ for
sufficiently large $\nu$ have finite supports and the degrees are well defined. Recall that
for a chain $A=\sum_{j=1}^d\alpha_j\{a_j\},$ $deg(A)=\sum_{j=1}^d\alpha_j$.)

The following lemma from \cite{Tw2} asserts that if $|Z_\nu|\to|Z|$ then for convergence of
chains $Z_\nu\ccon Z$ it suffices that the condition (c2) be satisfied on a dense subset of
the regular locus of $|Z|$.

\begin{lemma}
\label{eqconv}
Let $n\in\N$, and let $Z$ and $\{Z_\nu\}_{\nu\in\N}$ be positive $n$-chains in $U$.
If $|Z_\nu|\to|Z|$, then the following conditions are equivalent:
\begin{itemize}
\item[(i)] $Z_\nu\ccon Z$
\item[(ii)] For every point $a$ from a given dense subset of the regular
locus $Reg(|Z|)$, there exists a submanifold $T$ of $U$ of dimension $q-n$ transversal to
$|Z|$ at $a$ and such that $\overline{T}$ is compact, $|Z|\cap\overline{T}=\{a\}$ and
$deg(Z_\nu\cdot T)=deg(Z\cdot T)$ for almost all $\nu$.
\end{itemize}
\end{lemma}

Let now $X$ and $\{X_\nu\}_{\nu\in\N}$ be analytic sets of pure dimension $n$ in an open $U$
in $\C^q$. We say that the sequence $(X_{\nu})$ \emph{converges to $X$ in the sense of
(holomorphic) chains} when the sequence $(Z_\nu)$ of $n$-chains converges to the $n$-chain
$Z$, where $Z$ and $\{Z_\nu\}_{\nu\in\N}$ are obtained by assigning multiplicity $1$ to all
the irreducible components of $X$ and $\{X_\nu\}_{\nu\in\N}$ respectively.

\subsection{Holomorphic closure}
\label{sec:holclos}

Finally, let us recall the notion of holomorphic closure. Let $S$ be an arbitrary non-empty
subset of $\C^q$ and let $a\in\overline{S}$. By Noetherianity, there exists a smallest (with
respect to inclusion) germ of a complex analytic set at $a$ which contains the germ $S_a$. It
is called the \emph{holomorphic closure} of $S$ at $a$, and is denoted $\overline{S_a}^{HC}$.
For a systematic study of holomorphic closure of real analytic and semialgebraic sets, see
\cite{AS} and \cite{AR} respectively. Here, we shall only need the following fundamental
observation:

\begin{proposition}[{\cite[Prop.\,1]{AR}}]
\label{prop:AR}
The holomorphic closure of a semialgebraic set $S$ at a point $a\in\overline{S}$ is a Nash germ.
\end{proposition}

\section{Approximation of analytic sets by Nash sets}
\label{sec:main}

We prove in this section our main result, Theorem~\ref{thm:main}. We begin with an auxiliary
proposition, which can be viewed as a global variant of the R{\"u}ckert Lemma.

\begin{proposition}
\label{sdd} Let $\Omega$ be a domain of holomorphy in $\C^q$ and let $X$ be a nonempty
irreducible analytic subset of $\Omega$ of dimension $n<q$. Then, there are
$g_1,\ldots,g_{q-n}$, $h_1,\ldots,h_p\in\OO(\Omega)$ and a nowhere-dense subset $E$ of $X$
such that the following hold:
\begin{itemize}
\item[(i)] $X=\{g_1=\ldots=g_{q-n}=h_1=\ldots=h_p=0\}$
\item[(ii)] For every $a\in X\setminus E$ there is a neighbourhood $U$ of $a$ in $\Omega$ such
that
\[
\{g_1=\ldots=g_{q-n}=0\}\cap U=X\cap U
\]
and $(g_1,\ldots,g_{q-n})|_U$ is a submersion.
\end{itemize}
\end{proposition}

\begin{proof}
First let us prove the following\\[-3.5ex]

\subsection*{Claim} If $Z$ is an irreducible analytic subset of
$\Omega$ such that $X\varsubsetneq Z$, then for every
$\zeta\in\Reg(X)\cap\Reg(Z)$ there is $g\in\OO(\Omega)$ such that
$X\subset\{g=0\}$ and $(d_\zeta g)|_{T_\zeta\Reg(Z)}\neq0$.
\smallskip

Fix an arbitrary $\zeta\in\Reg(X)\cap\Reg(Z)$. By irreducibility of the sets $X$ and $Z$, one
has $\dim_zX<\dim_zZ$ for every $z\in X$. Therefore, there are a neighbourhood $V$ of $\zeta$
in $\Omega$ and $u\in\OO(V)$ such that $X\cap V\subset\{u=0\}$ and $(d_\zeta
u)|_{T_\zeta\Reg(Z)}\neq 0$.

Let $\mathcal{J}(X)$ denote the full sheaf of ideals of $X$ on $\Omega$ (see, e.g.,
\cite{GR}). By Cartan's Theorem A, the global sections $H^0(\Omega,\mathcal{J}(X))$ generate
$\mathcal{J}(X)_z$ for all $z\in\Omega$ (cf. \cite{GR}, p. 243). Consequently, there are
$f_1,\ldots,f_m\in\mathcal{O}(\Omega)$ with $X\subset\{f_1=\ldots=f_m=0\}$ and there are
$v_1,\ldots,v_m$ holomorphic in some neighbourhood of $\zeta$ in $\Omega$ such that
$u_\zeta=v_{1\zeta}f_{1\zeta}+\ldots+v_{m\zeta}f_{m\zeta}$. Since $(d_\zeta
u)|_{T_\zeta\Reg(Z)}\neq 0$, there must be $g\in\{f_1,\ldots,f_m\}$ such that $(d_\zeta
g)|_{T_\zeta\Reg(Z)}\neq 0$, which completes the proof of the claim.
\smallskip

Let us return to the proof of the proposition.
Without loss of generality, one can assume that $\Omega$ is connected.
Set $Z_1=\Omega$. Then, by Claim, there is $g_1\in\OO(\Omega)$
such that $X\subset\{g_1=0\}$ and $(d_{\zeta_1}g_1)|_{T_{\zeta_1}\Reg(Z_1)}\neq 0$ for
some $\zeta_1\in\Reg(X)\cap\Reg(Z_1)$.
It follows that $\zeta_1\in\Reg(X)\cap\Reg(Z_1\cap\{g_1=0\})$.

Let $Z_2$ be the irreducible component of $Z_1\cap\{g_1=0\}$ containing $X$. If $q-n>1$, then
$X\varsubsetneq Z_2$ and $\Reg(X)\cap\Reg(Z_2)\neq\varnothing$. By Claim, there exists
$g_2\in\OO(\Omega)$ such that $X\subset\{g_2=0\}$ and
$(d_{\zeta_2}g_2)|_{T_{\zeta_2}\Reg(Z_2)}\neq 0$ for some $\zeta_2\in\Reg(X)\cap\Reg(Z_2)$.
Then, again, $\zeta_2\in\Reg(X)\cap\Reg(Z_2\cap\{g_2=0\})$, and hence there exists a unique
irreducible component of $Z_2\cap\{g_2=0\}$ containing $X$; call it $Z_3$. If $q-n>2$, then
$X\varsubsetneq Z_3$ and $\Reg(X)\cap\Reg(Z_3)\neq\varnothing$ and Claim can be applied
again. We repeat this procedure until we get $Z_1, Z_2,\ldots, Z_{q-n}$ and
$g_1,g_2,\ldots,g_{q-n}$.

Clearly, $X$ is the unique irreducible component of $Z_{q-n}\cap\{g_{q-n}=0\}$ containing $X$.
Moreover, Theorem~\ref{zho22} implies that there are
$h_1,\ldots,h_p\in\OO(\Omega)$ such that $X=\{g_1=g_2=\ldots=g_{q-n}=h_1=\ldots=h_p=0\}$.
Set $E=\Sing(X)\cup\bigcup_{j=1}^{q-n}(E_j\cap X)$, where
$E_j=\Sing(Z_j)\cup\{z\in \Reg(Z_j):(d_zg_j)|_{T_z\Reg(Z_j)}=0\}$.
By construction, $g_1,\ldots,g_{q-n}$ and $E$ have all the required properties.
\end{proof}
\bigskip

We are now ready to prove our approximation theorem. Let $\Omega$ be a Runge domain in
$\C^q$, let $X$ be a complex analytic subset of $\Omega$ of pure dimension $n$, and let $Y$
be a Nash subset of $\Omega$ with $Y\subset X$.

\subsubsection*{Proof of Theorem~\ref{thm:main}}
Fix an open $\Omega_0\Subset\Omega$. We want to find a sequence $(X_\nu)_{\nu\in\N}$ of Nash
subsets of $\Omega_0$ of pure dimension $n$ converging to $X\cap\Omega_0$ in the sense of
chains and such that $X_\nu\supset Y\cap\Omega_0$ for all $\nu$.

Since irreducible components of $X$ can be approximated separately, $X$ can be assumed
irreducible. One can also assume that $X$ is nonempty and $X\varsubsetneq\Omega$. The proof
will be divided into three steps according to the properties of the Nash set $Y$.

\subsection*{Step 1} First, suppose that $Y=\varnothing$.\\
By Proposition~\ref{sdd}, we can fix $g_1,\ldots,g_{q-n},h_1,\ldots,h_p\in\OO(\Omega)$ and a
nowhere-dense subset $E$ of $X$ with the following properties:
\begin{itemize}
\item $X=\{g_1=\ldots=g_{q-n}=h_1=\ldots=h_p=0\}$
\item For every $a\in X\setminus E$, there is a neighbourhood $U$ of $a$ in $\Omega$ such that
$\{g_1=\ldots=g_{q-n}=0\}\cap U=X\cap U$ and $(g_1,\ldots,g_{q-n})|_U$ is a submersion.
\end{itemize}

Let $\Theta$ be the union of all irreducible components of $\{g_1=\ldots=g_{q-n}=0\}$
different from $X$. By Proposition~\ref{zhorm}, one can choose $f\in\OO(\Omega)$ such that
\begin{equation}
\label{eq:f}
\Theta\subset\{f=0\} \mathrm{\ \,and\ } f \mathrm{\ does\ not\ vanish\ identically\ on\ } X.
\end{equation}
Let $\Omega'\Subset\Omega$ be a Runge domain with $\Omega_0\Subset\Omega'$.
By Corollary~\ref{nstzatz}, there is a positive integer $t$ and there are $\alpha_{i,j}\in\OO(\Omega')$ such that
\begin{equation}
\label{eq:1}
f^th^t_i=\sum_{j=1}^{q-n}\alpha_{i,j}g_j,\mbox{ for } i=1,\ldots,p\,,
\end{equation}
on $\Omega'$. Theorem \ref{lat} implies that there are sequences $(f_\nu), (h_{i,\nu}),
(g_{j,\nu})$, and $(\alpha_{i,j,\nu})$ of Nash functions on $\Omega_0$ approximating
$f|_{\Omega_0}, h_i|_{\Omega_0}, g_j|_{\Omega_0}$, and $\alpha_{i,j}|_{\Omega_0}$,
respectively, and such that, for every $\nu\in\N$,
\begin{equation}
\label{eq:2}
f^t_\nu h^t_{i,\nu}=\sum_{j=1}^{q-n}\alpha_{i,j,\nu}g_{j,\nu},\mbox{ for } i=1,\ldots,p\,.
\end{equation}

For $\nu\in\N$, define
$\tX_\nu=\{g_{1,\nu}=\ldots=g_{q-n,\nu}=h_{1,\nu}=\ldots=h_{p,\nu}=0\}$, and let $X_\nu$ be
the $n$-dimensional part of $\tX_\nu$ (i.e., the union of all the $n$-dimensional irreducible
components of $\tX_\nu$). We shall prove that $(X_\nu)_{\nu\in\N}$ converges to
$X\cap\Omega_0$ in the sense of chains.

First let us show that $(X_\nu)$ converges to $X\cap\Omega_0$ locally uniformly
(i.e., that (l1) and (l2) of Section~\ref{sec:holchai} are satisfied). The condition (l2) is immediate.
Indeed, for any compact $K\subset\Omega_0$ with $K\cap X=\varnothing$, we have
$\inf_{z\in K}(\sum_{i=1}^{q-n}|g_i(z)|+\sum_{j=1}^p|h_j(z)|)>0$.
Consequently, $\inf_{z\in K}(\sum_{i=1}^{q-n}|g_{i,\nu}(z)|+\sum_{j=1}^p|h_{j,\nu}(z)|)>0$ for $\nu$ large enough,
hence $\tX_\nu\cap K=\varnothing$.

As for (l1), the density of $X\setminus(E\cup\{f=0\})$ in $X$ implies that it suffices to
check whether for every $a\in X\setminus (E\cup\{f=0\})$ there is a sequence $(a_\nu)$ with
$a_\nu\in X_\nu$ and $\lim_{\nu\to\infty}a_\nu=a$. Fix $a\in X\setminus(E\cup\{f=0\})$, and
set $X'_\nu:=\{g_{1,\nu}=\ldots=g_{q-n,\nu}=0\}$, for $\nu\in\N$. Recall that
$(g_1,\ldots,g_{q-n})$ is a submersion in some neighbourhood $V$ of $a$ such that $X\cap
V=\{g_1=\ldots=g_{q-n}=0\}\cap V$. Shrinking $V$ if needed, one can assume that
$(g_{1,\nu},\ldots,g_{q-n,\nu})|_V$ also is a submersion, for almost all $\nu$. For such
$\nu$, $X'_\nu\cap V$ is then an $n$-dimensional complex analytic manifold and one can
readily find a sequence $(a_\nu)$, $a_\nu\in X'_\nu$, convergent to $a$. Notice however that
$X'_\nu\cap V=X_\nu\cap V$ for $\nu$ large enough: Clearly, $X_\nu\cap V\subset X'_\nu\cap V$
for any $\nu$. On the other hand, it follows from \eqref{eq:f} that, for $\nu$ large enough,
we have $f_\nu|_{X'_\nu}\neq 0$. For every such $\nu$, the $h_{i,\nu}$ ($i=1,\dots,p$) vanish
identically on $X'_\nu\cap V$, by \eqref{eq:2}, and hence $X'_\nu\cap V\subset X_\nu\cap V$,
as required.

To complete the proof of the theorem in the case when $Y=\varnothing$, it remains to check
that $X\cap\Omega_0$ and $(X_\nu)_{\nu\in\N}$ satisfy condition (c2) of
Section~\ref{sec:holchai}. By Lemma~\ref{eqconv}, it suffices to show that, for every $a\in
X\setminus(E\cup\{f=0\})$, there exists a submanifold $T$ of $\Omega_0$ of dimension $q-n$
transversal to $X$ at $a$ and such that $\overline{T}$ is compact, $X\cap\overline{T}=\{a\}$
and $X_\nu\cap T$ is a singleton for all but finitely many $\nu$. Fix a point $a\in
X\setminus(E\cup\{f=0\})$. Choose a neighbourhood $V$ of $a$ in the same way as in the
previous paragraph. Let $L$ be the $q-n$ dimensional affine subspace of $\C^q$ normal to $X$
at $a$. It is now clear that any sufficiently small open ball $T\subset V\cap L$ centered at
$a$ satisfies the above requirements.

\subsection*{Step 2} Suppose now that $Y$ is a Nash set of pure dimension $d$.\\
Shrinking $\Omega$ if necessary, one can assume that $Y$ has finitely many irreducible components.
We shall construct a system of equations such that if
$g_{1,\nu},\ldots,g_{q-n,\nu}$, and $h_{1,\nu},\ldots,h_{p,\nu}$, in addition to \eqref{eq:2},
satisfy that system (together with some other functions), then $Y\cap\Omega_0\subset X_\nu$.

By Theorem \ref{chana}, one can choose an algebraic subset $Z$ of $\C^q$ of pure dimension
$d$ such that $Y$ is the union of certain analytic irreducible components of $Z\cap\Omega$.
Let $\Sigma$ be the union of all analytic irreducible components of $Z\cap\Omega$ which are
not contained in $Y$. By Proposition~\ref{zhorm}, there is $\bar{f}\in\OO(\Omega)$ such that
\begin{multline}
\label{eq:bf}
\Sigma\subset\{\bar{f}=0\} \mathrm{\ \,and\ } \bar{f} \mathrm{\ does\ not\ vanish\ identically}\\
 \mathrm{on\ any\ irreducible\ component\ of\ } Y.
\end{multline}
Let $\bar{g}_1,\ldots,\bar{g}_r$ be polynomials in $q$ complex variables such that
$Z=\{\bar{g}_1=\ldots=\bar{g}_r=0\}$. Let $g_1,\ldots,g_{q-n},h_1,\ldots,h_p$ be the
functions describing $X$ and let $\Omega'\Subset\Omega$ be a Runge domain such that
$\Omega_0\Subset\Omega'$, as in Step 1 of the proof.

Corollary~\ref{nstzatz} implies that there is an integer $u$ and there are $\beta_{k,j},
\gamma_{k,i}\in\OO(\Omega')$ such that, in addition to \eqref{eq:1}, we have
\begin{align}
\notag
{\bar{f}}^u g^u_j &= \sum_{k=1}^{r}\beta_{k,j}\bar{g}_k,\mbox{ for } j=1,\ldots,q-n,\mathrm{\ and}\\
\notag
{\bar{f}}^u h^u_i &= \sum_{k=1}^{r}\gamma_{k,i}\bar{g}_k,\mbox{ for } i=1,\ldots,p.
\end{align}
By Theorem~\ref{lat}, there are sequences $(f_\nu)$, $(h_{i,\nu})$, $(g_{j,\nu})$,
$(\alpha_{i,j,\nu})$, $(\bar{f}_\nu)$, $(\beta_{k,j,\nu})$, and $(\gamma_{k,i,\nu})$
of Nash functions on $\Omega_0$ approximating $f|_{\Omega_0}$, ${h_i}|_{\Omega_0}$,
${g_j}|_{\Omega_0}$, $\alpha_{i,j}|_{\Omega_0}$, $\bar{f}|_{\Omega_0}$,
$\beta_{k,j}|_{\Omega_0}$, and $\gamma_{k,i}|_{\Omega_0}$,
respectively, such that
\begin{align}
\label{eq:4}
{\bar{f}}^u_\nu g^u_{j,\nu} &= \sum_{k=1}^{r}\beta_{k,j,\nu}\bar{g}_k,\mbox{ for } j=1,\ldots,q-n,\mathrm{\ and}\\
\notag
{\bar{f}}^u_\nu h^u_{i,\nu} &= \sum_{k=1}^{r}\gamma_{k,i,\nu}\bar{g}_k,\mbox{ for } i=1,\ldots,p.
\end{align}
By Step 1, one can also assume that \eqref{eq:2} holds for every $\nu\in\N$.

Then, as proved in Step 1, the $n$-dimensional part $X_\nu$ of
$\tX_\nu=\{g_{1,\nu}=\ldots=g_{q-n,\nu}=h_{1,\nu}=\ldots=h_{p,\nu}=0\}$ approximates
$X\cap\Omega_0$ in the sense of chains. Moreover, by \eqref{eq:4},
$Y\cap\Omega_0\subset\tX_\nu$, but we still do not know whether $Y\cap\Omega_0\subset X_\nu$
for almost all $\nu\in\N$. To ensure the latter inclusion we shall need some additional
polynomial relations fulfilled.

First observe that $Y\cap\Omega_0$ can be assumed to have a finite number of analytic irreducible components
(if this were not the case, one can replace $\Omega_0$ with a slightly larger polynomial polyhedron
which is a relatively compact subset of $\Omega'$).
Next, after a linear change of coordinates in $\C^q$ if needed, one can assume that
$Z$ as a subset of \,$\C^q=\C^n\times\C^{q-n}$ has proper projection onto $\C^n$,
and that for every analytic irreducible component $\tY$ of $Y\cap\Omega_0$
there is a polydisc $D_1\times D_2\subset\Omega_0\subset\C^n\times\C^{q-n}$ such that
$\tY\cap(D_1\times D_2)\neq\varnothing$
and $X\cap(D_1\times D_2)$ has proper projection onto $D_1$.
\smallskip

Let $\hZ$ denote the image of the projection of $Z$ onto $\C^n$. Clearly, $\hZ$ is an
algebraic set of pure dimension $d$ and there are polynomials $w_1,\ldots,w_s$ in $n$ complex
variables such that $\hZ=\{w_1=\ldots=w_s=0\}$. Let $\hX\subset\Omega$ denote the
$d$-dimensional part of $X\cap(\hZ\times\C^{q-n})$.

By Step 1 of the proof, we know that $\hX\cap\Omega_0$ can be approximated by Nash subsets of
$\Omega_0$ of pure dimension $d$. More precisely, there are polynomials $Q_1,\ldots,Q_N$ and
there are $F_1,\ldots,F_M,\tau_1,\ldots,\tau_R\in\OO(\Omega)$ such that:
\begin{itemize}
\item $\hX=\{F_1=\ldots=F_M=0\}$ and $Q_j(F_1,\ldots,F_M,\tau_1,\ldots,\tau_R)=0$ for $j=1,\ldots,N$.
\item For all sequences $(F_{1,\nu}),\ldots,(F_{M,\nu})$, $(\tau_{1,\nu}),\ldots,(\tau_{R,\nu})$
of Nash functions on $\Omega_0$ approximating the restrictions
$F_1|_{\Omega_0},\ldots,F_M|_{\Omega_0}$, $\tau_1|_{\Omega_0},\ldots,\tau_{R}|_{\Omega_0}$, respectively,
and such that $Q_j(F_{1,\nu},\ldots,F_{M,\nu},\tau_{1,\nu},\ldots,\tau_{R,\nu})=0$ for
$j=1,\ldots,N,$ the $d$-dimensional parts of the sets $\{F_{1,\nu}=\ldots=F_{M,\nu}=0\}$
approximate $\hX\cap\Omega_0$ in the sense of chains.
\end{itemize}

By Proposition~\ref{zhorm}, one can choose $\bar{F}\in\OO(\Omega)$ with the following properties:
$\bar{F}$ vanishes identically on every irreducible component $S$ of
$X\cap(\hZ\times\C^{q-n})$ with $S\not\subset\hX$, and $\bar{F}$
does not vanish identically on any irreducible component of $\hX$.
Corollary~\ref{nstzatz} then implies that there is an integer $v$ and there are
$G_{l,j},H_{l,i},W_{l,k}\in\OO(\Omega')$ such
\begin{equation}
\notag
\bar{F}^vF^v_l=\sum_{j=1}^{q-n}G_{l,j}g_j+\sum_{i=1}^pH_{l,i}h_i+\sum_{k=1}^sW_{l,k}w_k,\mbox{ for }l=1,\ldots,M\,.
\end{equation}

By Theorem~\ref{lat}, there are Nash approximations of all the functions which we have
approximated so far satisfying \eqref{eq:2} and \eqref{eq:4}, and there are Nash
approximations $(\bar{F}_\nu)$, $(F_{l,\nu})$, $(\tau_{1,\nu}),\ldots,(\tau_{R,\nu})$,
$(G_{l,j,\nu})$, $(H_{l,i,\nu})$, and $(W_{l,k,\nu})$ of  $\bar{F}|_{\Omega_0}$, $F_l|_{\Omega_0}$,
$\tau_{1}|_{\Omega_0},\ldots,\tau_{R}|_{\Omega_0}$, $G_{l,j}|_{\Omega_0}$,
$H_{l,i}|_{\Omega_0}$, and $W_{l,k}|_{\Omega_0}$, respectively, such that
\[
Q_j(F_{1,\nu},\ldots,F_{M,\nu},\tau_{1,\nu},\ldots,\tau_{R,\nu})=0 \mbox{ for }j=1,\ldots,N,
\]
and
\begin{equation}
\label{eq:6}
{\bar{F}}^v_\nu F^v_{l,\nu}=
\sum_{j=1}^{q-n}G_{l,j,\nu}g_{j,\nu}+\sum_{i=1}^pH_{l,i,\nu}h_{i,\nu}+\sum_{k=1}^sW_{l,k,\nu}w_k,\mbox{for }l=1,\ldots,M.
\end{equation}

We claim now that $Y\cap\Omega_0$ is contained in the $n$-dimensional part $X_\nu$ of
$\tX_\nu=\{g_{1,\nu}=\ldots=g_{q-n,\nu}=h_{1,\nu}=\ldots=h_{p,\nu}=0\}$ for $\nu$ large
enough. Suppose that this is not the case. Then, one can choose an irreducible component
$\tY$ of $Y\cap\Omega_0$ which is not contained in $X_\nu$ for infinitely many $\nu\in\N$.

For this $\tY$, choose a polydisc $D_1\times D_2\subset\Omega_0\subset\C^n\times\C^{q-n}$ as
above. Then $\tY\cap(D_1\times D_2)\neq\varnothing$ and $X\cap(D_1\times D_2)$ has proper
projection onto $D_1$. Shrinking $D_1$ and $D_2$ if needed, one can additionally assume that
$\hZ\cap D_1$ is a connected $d$-dimensional complex analytic manifold, that
$X\cap(\hZ\times\C^{q-n})\cap(D_1\times D_2)=\tY\cap(D_1\times D_2)$, and that
$\tY\cap(D_1\times D_2)$ is a manifold such that over every point in $\hZ\cap D_1$ there is
precisely one point in $\tY\cap(D_1\times D_2)$.

Finally, we may assume that $\inf_{D_1\times D_2}|\bar{F}|>0$.
Indeed, since $\tY\cap(D_1\times D_2)$ is of pure dimension $d$, it follows that
$X\cap(\hZ\times\C^{q-n})\cap(D_1\times D_2)=\hX\cap(D_1\times D_2)$.
But $\bar{F}$ does not vanish identically on any irreducible component of $\hX$,
so shrinking $D_1\times D_2$ we get $\inf_{D_1\times D_2}|\bar{F}|>0$.

Set $\tilde{E}_\nu:=\tX_\nu\cap(\hZ\times\C^{q-n})\cap(D_1\times D_2)$.
Now, by \eqref{eq:4}, we know that $\tY\cap(D_1\times D_2)\subset \tilde{E}_\nu$ for $\nu$ large enough.
Note also that each $\tX_\nu$ (for $\nu$ large enough) has irreducible components of dimension at most $n$.
Therefore, if $\tY\not\subset X_\nu$, then
over a generic point in $\hZ\cap D_1$ there are at least two points in $\tilde{E}_\nu$.

On the other hand, by the fact that $\inf_{D_1\times D_2}|\bar{F}_{\nu}|>0$ for $\nu$ large enough,
and by \eqref{eq:6}, we have $\tilde{E}_\nu\subset\{F_{1,\nu}=\ldots=F_{M,\nu}=0\}\cap(D_1\times D_2)$.
Consequently, the $d$-dimensional part of $\{F_{1,\nu}=\ldots=F_{M,\nu}=0\}\cap(D_1\times D_2)$
does not converge to $\{F_1=\ldots=F_M=0\}\cap(D_1\times D_2)=\tY\cap(D_1\times D_2)$
in the sense of chains; a contradiction.

\subsection*{Step 3} Now, let $Y$ be an arbitrary Nash subset of $\Omega$ contained in $X$.\\
Let $Y=Y_0\cup\ldots\cup Y_{\dim(Y)}$ be the equidimensional decomposition of $Y$. Fix
$j\in\{0,\ldots,\dim(Y)\}$. By Step 2 of the proof, there is a system ($T_j$) of polynomial
equations satisfied by a system ($s$) of holomorphic functions, such that:
\begin{itemize}
\item ($s$) contains the functions $g_1,\ldots,g_{q-n},h_1,\ldots,h_p$
\item For every system ($s_{\nu}$) of Nash functions sufficiently close to
those from ($s$) on $\Omega_0$, and satisfying ($T_j$), the
following holds:
The $n$-dimensional part $X_\nu$ of
$\tX_\nu=\{g_{1,\nu}=\ldots=g_{q-n,\nu}=h_{1,\nu}=\ldots=h_{p,\nu}=0\}$ approximates
$X\cap\Omega_0$ and $Y_j\subset X_\nu$ ($\nu\in\N$), where
$g_{1,\nu},\ldots,g_{q-n,\nu},h_{1,\nu},\ldots,h_{p,\nu}$ are Nash functions from ($s_{\nu}$)
approximating $g_1,\ldots,g_{q-n},h_1,\ldots,h_p$, respectively, on $\Omega_0$.
\end{itemize}

Theorem~\ref{lat} allows one to approximate the systems of functions simultaneously for all
$j\in\{0,\ldots,\dim(Y)\}$ in such a way that the corresponding systems of polynomial
equations are satisfied, which implies that we can obtain the sequence $(X_\nu)_{\nu\in\N}$
with all the required properties. \qed

\section{Geometric obstruction to Nash approximation along a subset}
\label{sec:obstruction}

In the present section, we study the question of Nash approximation of a complex analytic set
along its arbitrary subset. Our interest originally developed from considerations of pairs
$(X,R)$ of a complex analytic set $X$ and its real analytic subset $R$, but it turned out
that the approximation question is independent of the real analytic structure of $R$.
Instead, it depends on the holomorphic closure of $R$. The following proposition
characterizes a local geometric obstruction to Nash approximation.

\begin{proposition}
\label{prop:local}
Let $X$ be a complex analytic subset of an open set $U$ in $\C^q$, and let
$R$ be an arbitrary subset of $X$. For every point $a\in R$, the following conditions are
equivalent:
\begin{itemize}
\item[(i)] There is an open neighbourhood $V$ of $a$ in $U$ and a sequence $(X_\nu)_{\nu=1}^\infty$
of Nash sets in $V$ locally uniformly convergent to $X\cap V$ and such that $X_\nu\supset R\cap V$ for all $\nu$.
\item[(ii)] There is a semialgebraic germ $S_a$ at $a$ such that $R_a\subset S_a\subset X_a$.
\item[(iii)] There is a Nash germ $Y_a$ at $a$ such that $R_a\subset Y_a\subset X_a$.
\end{itemize}
\end{proposition}

\begin{proof}
For the implication $\textrm{(i)}\Rightarrow\textrm{(iii)}$, suppose that $V$ is an open
neighbourhood of $a$ in $U$ and $(X_\nu)_{\nu=1}^\infty$ is a sequence of Nash sets in $V$
convergent to $X\cap V$ locally uniformly on $V$, and such that $X_\nu\supset R\cap V$ for
all $\nu$. We will show that there exists a Nash set $Y$ in $V$ such that $R_a\subset
Y_a\subset X_a$.

For a proof by contradiction suppose there is no such $Y$. Then the smallest (with respect to
inclusion) Nash subset of $V$ containing $R\cap V$, call it $Z$, is not a subset of $X\cap
V$. On the other hand, for every $\nu$, $X_\nu\cap Z$ is a Nash subset of $V$ containing
$R\cap V$. We claim that, for $\nu$ large enough, $X_\nu\cap Z$ is a proper subset of $Z$,
thus contradicting the minimality of $Z$. Indeed, choose a point $b\in Z\setminus X$. The set
$X$ being closed in $U$, there is a positive $\epsilon$ such that $B_\epsilon(b)\cap
X=\varnothing$, where $B_\epsilon(b)$ denotes the Euclidean ball with radius $\epsilon$
centered at $b\in\C^q$. The problem being local, one can assume without loss of generality
that $(X_\nu)$ converges to $X\cap V$ in the sense of Hausdorff metric (cf. \cite{Munkres}
and Section~\ref{sec:holchai}). Hence, for $\nu$ large enough, the Hausdorff distance between
$X\cap V$ and $X_\nu$ is less than $\epsilon/2$, and so $B_{\frac{\epsilon}{2}}(b)\cap
X_\nu=\varnothing$ for all such $\nu$. In particular, $b\in Z\setminus X_\nu$ so that $X_{\nu}\cap Z\varsubsetneq Z$.
\smallskip

The implication $\textrm{(iii)}\Rightarrow\textrm{(ii)}$ is immediate, since every Nash germ
is a germ of a semialgebraic set, by \cite[Prop.\,8.1.8]{BCR}.
\smallskip

As for $\textrm{(ii)}\Rightarrow\textrm{(i)},$ suppose that there is a semialgebraic germ
$S_a$ at $a$ such that $R_a\subset S_a\subset X_a$. By Proposition~\ref{prop:AR}, the
holomorphic closure of $S_a$ is a Nash germ, say $Y_a$. Since, by definition, $Y_a$ is the
smallest complex analytic germ containing $S_a$, it follows that $Y_a\subset X_a$. In other
words, locally near $a$, $R$ is contained in a Nash subset $Y$ of $X$. Now, by Theorem
\ref{thm:main} (in fact, here it suffices to use \cite[Thm.\,1.1]{B-IM}), there exists a
neighbourhood $V$ of $a$ in $U$ and a sequence $(X_\nu)_{\nu=1}^\infty$ of Nash subsets of
$V$ convergent locally uniformly to $X\cap V$ and such that $X_\nu\supset Y\cap V$ for all
$\nu$; hence property (i) holds.
\end{proof}
\medskip

Let us now turn to the proof of Theorem~\ref{thm:global-rel}. As asserted by the theorem, the
global obstruction to Nash approximation of a complex analytic set along its arbitrary subset is
of the same nature as the local one.

\subsubsection*{Proof of Theorem~\ref{thm:global-rel}}

We  shall prove that the following sequence of implications holds:
$\mathrm{(i)}\Rightarrow\mathrm{(ii)}\Rightarrow\mathrm{(v)}\Rightarrow\mathrm{(vi)}\Rightarrow\mathrm{(iv)}\Rightarrow
\mathrm{(iii)}\Rightarrow\mathrm{(i)}.$

The implications $\mathrm{(i)}\Rightarrow\mathrm{(ii)}$ and $\mathrm{(v)}\Rightarrow\mathrm{(vi)}$ are trivial.
As for $\mathrm{(vi)}\Rightarrow\mathrm{(iv)}$, it is an immediate consequence of the fact that
every Nash germ is a germ of a semialgebraic set, by \cite[Prop.\,8.1.8]{BCR}.

For the proof of $\mathrm{(iv)}\Rightarrow\mathrm{(iii)}$, fix a relatively compact open subset $\Omega_0$ of $\Omega$.
Under the assumptions of (iv) one can find a finite open covering $\{V_1,\dots,V_s\}$ of $\Omega_0$
such that, for every $j=1,\dots,s$, $V_j$ contains a subset
$S_j$ semialgebraic and satisfying $R\cap V_j\subset S_j\subset X\cap V_j$. Then
$S=S_1\cup\dots\cup S_s$ is a required semialgebraic set.

The implication $\mathrm{(iii)}\Rightarrow\mathrm{(i)}$ follows from Theorem~\ref{thm:main}
and Proposition~\ref{prop:semialg-Nash} below. Finally, for the proof of
$\mathrm{(ii)}\Rightarrow\mathrm{(v)}$, we proceed by contradiction:
Suppose that there is an open $\Omega_0\Subset\Omega$ such that
no Nash subset of $\Omega_0$ which contains $R\cap\Omega_0$ is contained in $X$.
Let $Z$ be the smallest (with respect to inclusion) Nash subset of $\Omega_0$ which contains $R\cap\Omega_0$.
Then $Z\not\subset X\cap\Omega_0$.
Now, let $(X_\nu)_{\nu=1}^\infty$ be a sequence of Nash subsets
of $\Omega_0$ that exists by property (ii) of the theorem. As in the proof of
Proposition~\ref{prop:local}, one shows that for $\nu$ large enough the set $X_\nu\cap Z$ is
a Nash subset of $\Omega_0$ containing $R\cap\Omega_0$ and properly contained in $Z$. This
contradicts the choice of $Z$, which completes the proof of the theorem. \qed
\medskip

The following result is a global analogue of Proposition~\ref{prop:AR}.

\begin{proposition}
\label{prop:semialg-Nash}
Let $\Omega$ be an open subset of $\C^q$ and let $X$ be complex analytic in $\Omega$. For every semialgebraic $S$ contained in $X$, there exists a Nash set $Y$ in $\Omega$ such that $S\subset Y\subset X$.
\end{proposition}

\begin{proof}
By semialgebraic stratification (Remark~\ref{rem:SA-facts}(2)), it suffices to consider the case when $S$ is a semialgebraic connected real analytic closed submanifold of an open subset of $\Omega$.

Then, by Proposition~\ref{prop:AR}, there exist an open set $V\subset\Omega$ and an irreducible Nash subset $N$ of $V$, such that $S\cap V$ is a nonempty connected manifold and $S\cap V\subset N\subset X$. By Theorem~\ref{chana}, one can choose an irreducible complex algebraic set $Z$ in $\C^q$ such that $N\subset Z$ and $\dim{N}=\dim{Z}$. By irreducibility of $N$, there exists a unique analytic irreducible component of $Z\cap\Omega$ which contains $N$; call it $Y$. Then, by Theorem~\ref{chana} again, $Y$ is a Nash set.
Moreover, we have $S\subset Y\subset X$. 
Indeed, $Y$ is an irreducible complex analytic set whose nonempty open subset
contains a nonempty open subset of $S$ and is contained in $X.$ The inclusions thus follow from the Identity Principle.
\end{proof}

\begin{remark}
\label{rem:global-holo-clos} The above proposition is closely related to the problem of
finding a \emph{global} holomorphic closure of a given set. It is important to notice that in
Proposition~\ref{prop:semialg-Nash}, one cannot expect that
$\overline{S_a}^{HC}=Y_a$ at every point $a\in S$. That is, in general, it may happen that
$\overline{S_a}^{HC}$ is a proper subgerm of $Y_a$ at some $a\in S$, for every choice of a
Nash set $Y$ satisfying the conclusion of the proposition, as can be seen in the following
example. On the other hand, if in addition $S$ is coherent, then its global holomorphic
closure exists in some neighbourhood of $S$ (that is, possibly after shrinking $\Omega$), by
\cite[Thm.\,1.1]{S}.
\end{remark}

\begin{example}
\label{ex:no-global-clos}
Let $\vp:\C\setminus\{\pm i\}\to\C^2$ be given as
\[
\vp(\zeta)=\left(\frac{\zeta^2-1}{\zeta^2+1}\,,\ \frac{2\zeta(\zeta^2-1)}{(\zeta^2+1)^2}\,\right)\,.
\]
Set $S':=\{\zeta=x+iy:(x-4/3)^2+y^2=1/9\}$, and $S:=\vp(S')$. Then $S$ is a connected irreducible (even smooth) semialgebraic real analytic set in $\C^2$, and since $\vp$ parametrizes the irreducible algebraic curve
\[
Y=\{(z,w)\in\C^2:z^4-z^2+w^2=0\}\,,
\]
it follows that $Y$ is the smallest Nash subset of $\Omega=\C^2$ that contains $S$. However, the germ $Y_{(0,0)}$ consists of two irreducible components and only one of them is the holomorphic closure of $S_{(0,0)}$.
\end{example}
\medskip

As an example of application of Proposition~\ref{prop:semialg-Nash} and Theorem~\ref{thm:global-rel}, we give here a criterion for existence of approximations of a holomorphic map along a given subset of its domain. For $w=(w_1,\ldots,w_m)\in\C^m,$ set $||w||=\max_{j=1,\ldots,m}|w_j|.$

\begin{theorem}
\label{thm:maps}
Let $U$ be a Runge domain in $\C^n$, let $F:U\to\C^m$ be a holomorphic mapping, and let $A$ be an arbitrary subset of $U$.
The following conditions are equivalent:
\begin{itemize}
\item[(i)] For every open $V\Subset U$ and $\epsilon>0$, there exists a Nash mapping $H:V\to\C^m$ which coincides with $F$ on $A\cap V$ and such that $\|H(z)-F(z)\|<\epsilon$ for all $z\in V$.
\item[(ii)] For every open $V\Subset U$, there exist a semialgebraic set $T$ with $A\cap V\subset T\subset U$ and a semialgebraic mapping $G:T\to\C^m$ such that $F|_T\equiv G$.
\end{itemize}
\end{theorem}

\begin{proof}
Set $\Omega:=U\times\C^m$, $X:=\{(z,F(z)):z\in U\}$, and $R:=\{(z,F(z)):z\in A\}$. Clearly, $X$ is of pure dimension $n$.
Given an open $V$ relatively compact in $U$, there is $M_V>0$ such that $\|F(z)\|<M_V$ for all $z\in V$, and so
\[
\Omega_V:=V\times\{w=(w_1,\dots,w_m)\in\C^m:|w_j|<M_V, j=1,\dots,m\}
\]
is a relatively compact open subset of $\Omega$.

With this terminology, condition (ii) of the theorem is now equivalent to saying that for every $\Omega_V$ as above there is a semialgebraic set $S$ such that $R\cap\Omega_V\subset S\subset X$. 
The implication $\mathrm{(ii)}\Rightarrow\mathrm{(i)}$ thus follows from Proposition~\ref{prop:semialg-Nash} and \cite[Thm.\,3.6]{TA}.
(Of course, it could be also derived from Theorem~\ref{thm:main}.)

Condition (i), on the other hand, implies that for every open $\Omega_0\Subset\Omega$, $X\cap\Omega_0$ can be approximated locally uniformly along $R\cap\Omega_0$ by a sequence $(X_\nu)_{\nu\in\N}$ of Nash subsets of $\Omega_0$ of pure dimension $n$.
Therefore, $\mathrm{(i)}\Rightarrow\mathrm{(ii)}$ is a consequence of the implication $\mathrm{(ii)}\Rightarrow\mathrm{(iii)}$ in Theorem~\ref{thm:global-rel}.
\end{proof}

Theorem~\ref{thm:maps} can be used, in particular, for Nash approximation of holomorphic extensions of a given map.
Indeed, the theorem characterizes those pairs $(F,f)$ of a holomorphic extension $F$ of a given map $f$ for which $F$ can be approximated by Nash mappings each of which extends $f$.


\section*{Acknowledgements}

We would like to thank Professor Marek Jarnicki for pointing out the division theorem in
Section~\ref{sec:divthjp} as an elegant tool for deriving the global Nullstellensatz on
domains of holomorphy. We are also grateful to Professor Rasul Shafikov for stimulating discussions.


\end{document}